\documentclass[12pt,a4paper,twoside,reqno]{amsart}
\usepackage[a4paper,top=3cm,bottom=3cm,inner=2.7cm,outer=2.7cm]{geometry}

\usepackage{amssymb,amsmath,amsthm}
\usepackage[english]{babel}
\usepackage{booktabs}
\usepackage{array}
\usepackage{mathtools}
\usepackage{enumitem}
\usepackage{xparse}
\usepackage{multirow,longtable}
\usepackage{pst-node,placeins,xspace,fix-cm}
\usepackage{url}
\usepackage[colorlinks,linkcolor=blue,citecolor=violet,urlcolor=darkgray,final,hyperindex,linktoc=page,hyperfootnotes=true]{hyperref}

\newenvironment{enumT}{\begin{enumerate}[label=$($\hspace{-0.1ex}\roman*\hspace{0.13ex}$)$]}{\end{enumerate}}

\newenvironment{eqD*}{\begin{equation*}}{\end{equation*}\ignorespacesafterend}

\newtheorem{theorem}{Theorem}[section]
\newtheorem{proposition}[theorem]{Proposition}
\newtheorem{lemma}[theorem]{Lemma}
\newtheorem{corollary}[theorem]{Corollary}

\theoremstyle{definition}

\newtheorem{remark}[theorem]{Remark}

\def\nameit#1{\textrm{#1}~}
\def\thex{\nameit{Theorem}}
\def\prox{\nameit{Proposition}}
\def\corx{\nameit{Corollary}}
\def\lemx{\nameit{Lemma}}

\def\remx{\nameit{Remark}}

\numberwithin{equation}{section}

\def\:{\colon}

\newcommand{\aR}[2][]{\xRightarrow[#1]{#2}}

\DeclareMathOperator{\Ann}{Ann}
\DeclareMathOperator{\Ass}{Ass}
\DeclareMathOperator{\Min}{Min}
\DeclareMathOperator{\Max}{Max}

\DeclareMathOperator{\Soc}{Soc}
\DeclareMathOperator{\Spec}{Spec}
\DeclareMathOperator{\depth}{depth}
\newcommand{\ccont}{\operatorname{c}}
\newcommand{\C}{\ccont}
\newcommand{\cont}{\subseteq}
\newcommand{\contain}{\supseteq}
\newcommand{\0}{0}
\def\set#1#2{\left\{{#1}\left.\right|\,{#2}\right\}}
\newcommand{\Z}{\mathbb{Z}}
\newcommand{\iso}{\cong}
\newcommand{\Prod}[2]{\underset{#1}{\prod}\hspace{0.1ex}{#2}}
\def\phi{\varphi}
\newcommand{\too}{\longrightarrow}

\begin{document}

\title[Detecting Essential Ideals in Polynomial Rings]{Detecting Essential Ideals in Polynomial Rings}

\author{Amartya Goswami}
\address{Department of Mathematics and Applied Mathematics, University of Johannesburg, P.O.\ Box 524, Auckland Park 2006, South Africa; National Institute for Theoretical and Computational Sciences (NITheCS), Johannesburg, South Africa}
\email{agoswami@uj.ac.za}

\author{Luca Mesiti}
\address{Department of Mathematical Sciences, Stellenbosch University, Private Bag X1, Matieland, Stellenbosch 7600, South Africa}
\email{luca.mesiti@outlook.com}


\date{}
\dedicatory{}

\begin{abstract}
We characterize essential ideals in polynomial rings $R[(X_\lambda)_{\lambda\in\Lambda}]$ over a commutative ring $R$ with $1$. We present several characterizations, letting $R$ vary among notable classes of rings. The idea is to detect essentiality by checking the intersections with principal ideals generated by polynomials in a test class. In the general commutative case, we apply McCoy’s zero-divisor criterion to construct an efficient test class of polynomials in terms of annihilators of content ideals. We then refine the test class and strengthen the result in several ways, assuming further properties on the coefficients ring $R$. Assuming that $R$ is Noetherian, we reduce the test class using the associated primes of $R$. When $R$ satisfies Serre's condition $(S_1)$, it suffices to use minimal primes. We also study the cases of $R$ Artinian, $R=\mathbb Z/n\mathbb Z$ and $R$ equal to the ideal-adic completion of an excellent ring, among others.
\end{abstract}

\maketitle
\enlargethispage{6pt}

\section{Introduction}\label{sec:introduction}

In this paper, we present criteria to detect essential ideals in polynomial rings $P=R[(X_\lambda)_{\lambda\in\Lambda}]$ with possibly infinitely many variables, letting $R$ vary among notable classes of rings. In general, all rings we consider are assumed to be nonzero, commutative, and unital.

An ideal $E$ of a ring $A$ is \emph{essential} if $E\cap J\ne0$ for every nonzero ideal $J$ of $A$; equivalently, $E$ is an essential submodule of the module $A$. Intuitively, an ideal is essential if it is so deeply woven into the ring that it overlaps with all other nonzero ideals. This makes essential ideals into excellent tools of topological nature that measure largeness inside the ring. A notable example is that essential extensions are fundamental in the construction of injective envelopes \cite{EckmannSchopf1953}. Moreover, Matlis theory relates essential ideals to associated primes in the Noetherian setting \cite{Matlis1958}. Since the 1970s, essential ideals have been studied intensely. In particular, descriptions of essential ideals are known for quotients of Pr\"ufer domains, rings of continuous functions,
incidence algebras, and matrix rings; see, for instance, 
\cite{AndruszkiewiczPryszczepko2026,Azarpanah1995,Azarpanah1997,GreenVanWyk1989,Spiegel2000,Taherifar2014,Taherifar2015}.
For polynomial rings, Armendariz's coefficient-annihilator formula in \cite{Armendariz1974} shows that, over a reduced ring, an ideal of $R[x]$ is essential exactly when the ideal
generated by all of its coefficients is essential in $R$. On the other hand, Krasula
proved in \cite{Krasula2022} that every quotient of $R[x]$ by an essential ideal is Artinian exactly when $R$ is reduced Artinian. These results, however, do not give an ideal-by-ideal criterion for essentiality in polynomial rings with arbitrary coefficient rings, particularly in the presence of nilpotents.

The inspiration for the present paper is that aggregate content no longer determines essentiality beyond the reduced setting. Recall that, for $f\in P$, the content $\ccont(f)$ of $f$ is the ideal generated by its coefficients. Although polynomial content is
the basic tool for the `Ohm--Rush theory' in \cite{OhmRush1972}, aggregate content does not keep the arrangement of coefficients within individual polynomials. We thus investigate criteria for essentiality of ideals in $P$ beyond the setting of a reduced ring of coefficients. We show that McCoy's theorem from \cite{McCoy1942} (see also
\cite{McCoy1957}) can be employed to yield an efficient test class of polynomials that detects the essentiality of ideals in $P$. More precisely, by McCoy's theorem, a nonzero
$f\in P$ is a zero divisor precisely when $\Ann_R(\ccont(f))\ne0$. And this characterization of zero divisors is precisely what we need to form a test class of polynomials $f$ to check against an ideal $E$ of $P$.

Assuming further properties on the ring $R$ of coefficients, we can then strengthen the result of characterization of essential ideals in $P=R[(X_\lambda)_{\lambda\in\Lambda}]$. For Noetherian coefficients, in Subsection~\ref{subsec:noeth}, we can greatly reduce the test class of polynomials $f$ using the associated primes of $R$ seen as an $R$-module. Then, in Subsection~\ref{subsec:s1}, we replace associated primes by minimal primes under Serre's condition $(S_1)$. This can be immediately applied to yield stronger characterization results for essential ideals in polynomial rings over a ring of coefficients that is either complete intersection, Noetherian Gorenstein, Noetherian Cohen--Macaulay or Noetherian reduced. For Artinian coefficients, in Section~\ref{sec:artinian}, we show that essential ideals in $P$ are determined by the socle. We then sharpen this in the case of Gorenstein Artinian coefficients. As an application, we present the case of polynomials over $\mathbb Z/n\mathbb Z$. Finally, in Section~\ref{sec:excellent}, we characterize essential ideals of polynomial rings over ideal-adic completions of excellent rings.

\section{McCoy tests and associated primes}\label{sec:general}





Our first theorem gives the basic criterion for essentiality of ideals in polynomial rings $P=R[(X_\lambda)_{\lambda\in\Lambda}]$, on which the subsequent refinements are based.

Note that, since every nonzero ideal of $P$ contains a nonzero principal ideal, we can always determine the essentiality of an ideal $E$ of $P$ by checking the intersections of $E$ with all principal ideals $(f)$ in $P$. We show that McCoy's theorem yields an efficient test class of polynomials $f$ to check against $E$ to determine whether $E$ is essential. 

\begin{theorem}\label{theorgencomm}
    Let $R$ be any commutative ring. Consider $E$ an ideal of $P = R[\{x_\lambda\}]$. The following are equivalent:
    \begin{enumT}
        \item $E$ is an essential ideal of $P$;
        \item $E\ne 0$ and $E\cap (f)\ne 0$ for every nonzero $f\in P$ such that $\Ann_R\ccont(f)\ne0$.
    \end{enumT}
\end{theorem}
\begin{proof}
Of course $(i)\aR{}(ii)$. We prove the converse. As we said above, $E$ is essential if and only if $E\ne 0$ and $E\cap (f)\ne 0$ for every nonzero $f\in P$. Moreover, observe that it is sufficient to consider only the zero-divisor polynomials $f\in P$. Indeed, if
$E\ne0$ and $f\in P$ is regular (\textit{i.e.}, a non-zero-divisor), then any nonzero element $e\in E$ gives
\[
0\ne ef\in E\cap (f).
\] 

Now, we apply McCoy's theorem (see \cite{McCoy1942,McCoy1957}) to characterize zero-divisor polynomials $f\in P$. For every nonzero $f\in P$, we claim that
\begin{equation}\label{eq:mccoy-criterion}
f\text{ is a zero divisor in }S
\;\text{if and only if}\;
\Ann_R\ccont(f)\ne0.
\end{equation}
For the case of finitely many indeterminates, this is exactly McCoy's theorem. Then, notice that although the set $\Lambda$ of available indeterminates may be infinite, each polynomial is a finite $R$-linear combination of monomials, and each monomial involves only finitely many indeterminates. Thus,  (\ref{eq:mccoy-criterion}) holds and we conclude.
\end{proof}

\subsection{Noetherian coefficients}\label{subsec:noeth}

Assuming that the ring $R$ is Noetherian, we can strengthen the result of \thex\ref{theorgencomm}. The test class of polynomials $f\neq 0$ reduces to those such that $\Ann_R(\C(f))$ is an associated prime of $R$. Again, it will not be a problem to consider infinitely many indeterminates. This is because the scalar multiplier that we construct in the proof lives in $R$.

\begin{theorem}\label{theornoetherian}
    Let $R$ be a Noetherian ring. Consider $E$ an ideal of $P = R[\{x_\lambda\}]$. The following are equivalent:
    \begin{enumT}
        \item $E$ is an essential ideal of $P$;
        \item $E \neq \0$ and $E \cap (f) \neq \0$ for all $f\neq 0$ such that $\Ann_R(\C(f)) \in \Ass(R)$.
    \end{enumT}
\end{theorem}
\begin{proof}
    Of course $(i)\aR{}(ii)$. We prove the converse. By \thex\ref{theorgencomm}, it suffices to prove that $E \cap (g) \neq \0$ for all $g\neq 0$ such that $\Ann_R(\C(g)) \neq 0$. So consider such a $g$. The strategy will be to construct, starting from $g$, a polynomial $f\in (g)$ that belongs to the test class described in $(ii)$.
    
    It is a well-known theorem that, for a Noetherian ring, every non-zero module has at least one associated prime. Since $R$ is Noetherian and $g$ only contains finitely many indeterminates, $g$ belongs to a polynomial subring $P'$ of $P$ that is Noetherian. Then the ideal $\langle g \rangle$ generated by $g$ in $P'$, which is contained in the ideal $(g)$ generated by $g$ in $P$, is a $P'$-module which must have at least one associated prime. So there exists a non-zero element $f\in \langle g \rangle \cont (g)$ such that $\Ann_{P'}(f)$ is a prime ideal of $P'$, and thus an associated prime of $P'$ (seen as a $P'$-module). Now, by \cite[1–7, Theorem 7 and Proposition 2.4, p.\ 4]{BH74} (see also \cite{BH74}), the associated primes of the polynomial ring $P'$ are exactly the extensions of the associated primes of $R$. Therefore, there exists some $\mathfrak{p} \in \text{Ass}(R)$ such that
    $$\Ann_{P'}(f) = \mathfrak{p}P'.$$

    We claim that then 
    $$\Ann_R(\C(f)) = \mathfrak{p}\in \Ass(R).$$
    Note that the content ideal $\C(f)$ does not depend on seeing $f$ as an element of $P'$ or of $P$. In order to prove the claim, consider $a\in \Ann_R(\C(f))$. Then $a$ is an element of $R$ that annihilates all coefficients of the polynomial $f$. Whence $a$, seen as an element of $P'$, annihilates $f$. Since $\Ann_{P'}(f)=\mathfrak{p}P'$ and $a\in R$, we conclude that $a\in \mathfrak{p}$. Starting instead with an element $r\in \mathfrak{p}$, we know that $r\in \Ann_{P'}(f)$ since $\Ann_{P'}(f)=\mathfrak{p}P$. But the only way in which $r\in R$, seen as an element of $P'$, can annihilate the polynomial $f$ is that $r$ annihilates the coefficients of $f$. Whence $r\in \Ann_R(\C(f))$.

    By the assumption $(ii)$, we obtain that
    $$E\cap (g)\contain E\cap (f)\neq \0.$$
    Whence we conclude that $E$ is an essential ideal of $P$.
\end{proof}

We also present a more constructive proof of \thex\ref{theornoetherian}. The idea comes from the hope that we can find a multiplier $b\in R$ such that $f:=bg\in (g)$ belongs to the test class described in $(ii)$ of \thex\ref{theornoetherian}. Note that the content ideals of $f$ and $g$ are then bound by the equation $\C(f)=b\C(g)$. 

\begin{proof}[Proof of \thex\ref{theornoetherian}, more constructively]
    Like in the proof above, start with $0\neq g\in P$ such that $\Ann_R(\C(g)) \neq 0$. We construct from $g$ a polynomial $f=b\cdot g$, with $b\in R$, that belongs to the test class described in $(ii)$.

    Consider the family of all annihilators of non-zero scalar multiples of the content ideal $\C(g)$:
    $$\mathcal{F}:=\set{\Ann_R(x \C(g))}{x\in R,\; x\C(g)\neq \0}$$
    Note that $\mathcal{F}$ is a non-empty family of ideals of $R$, as $\Ann_R(1\cdot \C(g))\in \mathcal{F}$. Then, since $R$ is Noetherian, $\mathcal{F}$ has a maximal element $\Ann_R(b\C(g))$ under inclusion.

    We prove that $\mathfrak{p}:=\Ann_R(b\C(g))$ is a prime ideal. Note first that $\mathfrak{p}\neq R$, because $1\notin \Ann_R(b\C(g))$, by construction of $\mathcal{F}$. Let $r,s\in R$ such that $rs\in \mathfrak{p}$ but $s\notin \mathfrak{p}$. By definition, $rsb\C(g)=\0$ but $sb\C(g)\neq \0$. But
    $$\mathfrak{p}=\Ann_R(b\C(g))\cont \Ann_R(sb\C(g))$$
    Thus, by maximality of $\mathfrak{p}$, the two annihilators must be equal. Since $r$ annihilates $sb\C(g)$, it must annihilate $b\C(g)$. Whence $r\in \mathfrak{p}$.

    We now prove that $\mathfrak{p}$ is an associated prime of $R$. Note that $\C(g)$ is finitely generated, by the non-zero coefficients $\{c_1, \ldots, c_k\}$ of $g$. Whence $b\C(g)=(bc_1,\ldots,bc_k)$. It follows that
    $$\mathfrak{p}=\Ann_R(b\C(g))=\bigcap_{i\in \{1,\ldots,k\}} \Ann_R(bc_i)$$
    Since $\mathfrak{p}$ is a prime ideal that is equal to a finite intersection of ideals, it must be equal to one of those ideals (see \cite[Proposition 1.11(ii), p.\ 8]{AtiyahMacdonald1969}). Say that $\mathfrak{p}=\Ann_R(bc_j)$. Moreover, $bc_j\neq 0$, since $\mathfrak{p}\neq R$. So $\mathfrak{p}\in \Ass(R)$.

    We conclude that $f=b\cdot g\in (g)$ is such that $f\neq 0$ (since $b\C(g)\neq \0)$ and
    $$\Ann_R(\C(f))=\Ann_R(b\C(g))=\mathfrak{p}\in \Ass(R).$$
    Whence $E\cap (g)\contain E\cap (f)\neq \0$.
\end{proof}

\begin{remark}\label{remarkcomparenoethclassgencomm}
    If $R$ is an integral domain, then the polynomial rings $P$ over $R$ are as well integral domains. And essential ideals of $P$ reduce to all non-zero ideals.

    So the theory developed in this paper is interesting when $R$ is not an integral domain. In these cases, the test class of polynomials $f$ described in condition $(ii)$ of \thex\ref{theornoetherian} is a subset of the test class described in condition $(ii)$ of \thex\ref{theorgencomm}. Indeed, $\0\notin \Ass(R)$ as $\0$ is not prime. In general, the test class described in condition $(ii)$ of \thex\ref{theornoetherian}, for Noetherian rings $R$, is a much smaller subset of the other test class. Note that for every Noetherian ring $R$, the set $\Ass(R)$ of associated primes of $R$ is finite.
\end{remark}

We can also observe the following, which further describes how the coefficients of the polynomials $f$ of the test class described in condition $(ii)$ of \thex\ref{theornoetherian} must be formed. Note that the assumption in the next proposition that $R$ is not an integral domain is not restrictive, as also explained in \remx\ref{remarkcomparenoethclassgencomm}.

\begin{proposition}\label{propcfcontained}
    Let $R$ be a Noetherian ring that is not an integral domain. Let $f\in P = R[\{x_\lambda\}]$ be a polynomial of the test class described in condition $(ii)$ of \thex\ref{theornoetherian}. That is, $f\neq 0$ and $\Ann_R(\C(f))\in \Ass(R)$. Then $f$ is such that $\C(f)\cont \mathfrak{p}$ for some associated prime ideal $\mathfrak{p}$ of $R$.
\end{proposition}
\begin{proof}
    Call $\mathfrak{q}:=\Ann_R(\C(f))\in \Ass(R)$. Since $R$ is not an integral domain, then $\mathfrak{q}\neq \0$. So $\C(f)$ cannot contain a regular element, which means that $\C(f)\cont \operatorname{Z}(R)$. Since $R$ is Noetherian, the set $\operatorname{Z}(R)$ of zero divisors of $R$ coincides with the union of the associated primes of $R$. Moreover, this is a finite union, again thanks to the fact that $R$ is Noetherian. By the prime avoidance lemma, we obtain that $\C(f)$ is entirely contained inside one of the associated primes of $R$.
\end{proof}

\subsection{The $(S_1)$ refinement}\label{subsec:s1}

The Noetherian criterion is formulated in terms of all associated primes of $R$, and therefore includes contributions from embedded associated primes. Under Serre's condition $(S_1)$, we can strengthen the criterion to consider just minimal primes.

Recall that Serre's condition $(S_1)$ for a Noetherian ring is
\[
 \depth R_{\mathfrak p}\ge
 \min\{1,\dim R_{\mathfrak p}\}
 \qquad(\mathfrak p\in\Spec R).
\]
This condition precisely guarantees that the ring has no embedded associated primes: all associated primes are minimal. Indeed, see \cite[Theorem 4.5.2, p.\ 73, especially its proof]{HunekeSwanson2006}. And we can use this to sharpen the result of \thex\ref{theornoetherian}.

The results of this subsection directly also apply to all the following cases:
\begin{itemize}
    \item $R$ complete intersection;
    \item $R$ Noetherian Gorenstein;
    \item $R$ Noetherian Cohen--Macaulay;
    \item $R$ Noetherian reduced
\end{itemize}
Indeed, complete intersection and Noetherian Gorenstein rings are in particular Noetherian Cohen--Macaulay rings. Moreover, Noetherian Cohen--Macaulay rings satisfy ($S_1$) by \cite[discussion before Theorem 4.5.2, p.\ 73]{HunekeSwanson2006} and Noetherian reduced rings satisfy ($S_1$) by \cite[discussion before Theorem 4.5.2, p.\ 73]{HunekeSwanson2006}.

\begin{theorem}\label{theorconditionS1}
    Let $R$ be a Noetherian ring that satisfies Serre's condition (S1). Consider $E$ an ideal of $P = R[\{x_\lambda\}]$. The following are equivalent:
    \begin{enumT}
        \item $E$ is an essential ideal of $P$;
        \item $E \neq \0$ and $E \cap (f) \neq \0$ for all $f\neq 0$ such that $\Ann_R(\C(f))\in \Min(R)$.
    \end{enumT}
\end{theorem}
\begin{proof}
    It is a well-known result that, if $R$ satisfies Serre's condition (S1), then
    $$\Ass(R)=\Min(R).$$
    See for example \cite[Theorem 4.5.2, p.\ 73, especially its proof]{HunekeSwanson2006}. We then conclude applying \thex\ref{theornoetherian}.
\end{proof}

Similarly to the general Noetherian case, we can then further describe how the coefficients of the polynomials $f$ of the test class described in condition $(ii)$ of \thex\ref{theorconditionS1} must be formed.

\begin{proposition}
    Let $R$ be a Noetherian ring that satisfies Serre's condition (S1) and is not an integral domain. Let $f\in P = R[\{x_\lambda\}]$ be a polynomial of the test class described in condition $(ii)$ of \thex\ref{theorconditionS1}. That is, $f\neq 0$ and $\Ann_R(\C(f))\in \Min(R)$. Then $f$ is such that $\C(f)\cont \mathfrak{p}$ for some minimal prime ideal $\mathfrak{p}$ of $R$.
\end{proposition}
\begin{proof}
    Since $R$ satisfies Serre's condition (S1), $\Ass(R)=\Min(R)$. Then the result follows from \prox\ref{propcfcontained}.
\end{proof}

\section{Artinian coefficients: socle and Gorenstein tests}\label{sec:artinian}

Assuming that the ring $R$ is Artinian, we can use the socle $\Soc(R)$ of $R$ to find a convenient test class of polynomials $f$ to check essentiality of ideals of $P = R[\{x_\lambda\}]$. We obtain that it just suffices to look at those polynomials $f$ whose coefficients are entirely generated by the socle of $R$. The key idea behind this criterion is that the Jacobson radical of an Artinian ring is nilpotent.

\begin{theorem}\label{theorartinian}
    Let $R$ be an Artinian ring. Consider $E$ an ideal of $P = R[\{x_\lambda\}]$. The following are equivalent:
    \begin{enumT}
        \item $E$ is an essential ideal of $P$;
        \item $E \neq \0$ and $E \cap (f) \neq \0$ for all $f\neq 0$ such that $\C(f)\cont \Soc(R)$.        
    \end{enumT}
\end{theorem}
\begin{proof}
    Of course $(i)\aR{}(ii)$. We prove the converse. By \thex\ref{theorgencomm}, it suffices to prove that $E \cap (g) \neq \0$ for all $g\neq 0$ such that $\Ann_R(\C(g)) \neq 0$. So consider such a $g$. The strategy will be to construct, starting from $g$, a polynomial $f\in (g)$ that belongs to the test class described in $(ii)$.
    
    Since $R$ is Artinian, by \cite[Theorem 4.12, p.\ 54]{Lam01}, its Jacobson radical $J$ is a nilpotent ideal. Then there exists an integer $k\ge 0$ such that $J^k \C(g)\neq \0$ but $J^{k+1} \C(g)=\0$. Indeed, $\C(g)\neq 0$ since $g\neq 0$, and some power of $J$ is guaranteed to be $\0$. We can thus pick an element $m\in J^k$ such that $m \C(g)\neq \0$. And we obtain that the polynomial $f:= mg$ is non-zero.

    Now, we observe that 
    $$J \C(f)= J m \C(g) \cont J^{k+1}\C(g) = \0.$$
    Which means that the coefficients of the polynomial $f$ are annihilated by the Jacobson radical $J$ of $R$.

    By \cite[Proposition 18.39]{Fai76}, in an Artinian ring, the set of elements annihilated by the Jacobson radical is exactly the socle. So we obtain that the coefficients of the polynomial $f$ must belong to $\Soc(R)$. That is, $\C(f)\cont \Soc(R)$.

    By the assumption $(ii)$, then,
    $$E\cap (g) \contain E\cap (f)\neq 0.$$
    Whence we conclude that $E$ is an essential ideal of $P$.
\end{proof}

\begin{remark}\label{remcompareconditionsii}
    \thex\ref{theorartinian} could also be derived from \thex\ref{theornoetherian}, using that every Artinian ring is in particular Noetherian. The main ingredients of such a proof are the following two properties of Artinian rings $R$:
    \begin{enumerate}
        \item associated primes precisely coincide with the maximal ideals, i.e.\
        $$\Ass(R)=\Max(R)$$
        (see \cite[Theorem 8.5, p.\ 90]{AtiyahMacdonald1969} and \cite[Theorem 6.5, p.\ 39]{Matsumura1986});
        \item the socle is built using the annihilators of the maximal ideals, as follows:
        $$\Soc(R)=\bigoplus_{\mathfrak{m}\in \Max(R)}\Ann_R(\mathfrak{m}).$$
        and this union is finite (see \cite[Theorem 4.3.7(i), p.\ 95]{Coh03});
    \end{enumerate}
\end{remark}

If we further know that the ring $R$ is Artinian local, the condition described in $(ii)$ of \thex\ref{theorartinian} becomes even easier to check. Indeed, we have the following result.

\begin{corollary}\label{corollartinianlocal}
    Let $R$ be an Artinian local ring, and call $\mathfrak{m}$ its unique maximal ideal. Consider $E$ an ideal of $P = R[\{x_\lambda\}]$. The following are equivalent:
    \begin{enumT}
        \item $E$ is an essential ideal of $P$;
        \item $E \neq \0$ and $E \cap (f) \neq \0$ for all $f\neq 0$ such that $\C(f)\cont \Ann_R(\mathfrak{m})$.     
    \end{enumT}
\end{corollary}
\begin{proof}
    Since $R$ is Artinian local, by \cite[Proposition 18.39]{Fai76} (or also by \cite[Theorem 4.3.7(i), p.\ 95]{Coh03}), its socle precisely coincides with the annihilator of its unique maximal ideal. We then conclude by applying \thex\ref{theorartinian}.
\end{proof}

Furthermore, when $R$ is Artinian local, we can strengthen \remx\ref{remcompareconditionsii}. We obtain that the test class of polynomials described in condition $(ii)$ of \thex\ref{theornoetherian} precisely coincides with the test class described in condition $(ii)$ of \corx\ref{corollartinianlocal}.

\begin{proposition}
    Let $R$ be an Artinian local ring, and call $\mathfrak{m}$ its unique maximal ideal. Consider $0\neq f\in P = R[\{x_\lambda\}]$. The following are equivalent:
    \begin{enumT}
        \item $\Ann_R(\C(f))\in \Ass(R)$;
        \item $\C(f)\cont \Ann_R(\mathfrak{m})$.
    \end{enumT}
\end{proposition}
\begin{proof}
    By \remx\ref{remcompareconditionsii}, we know that
    $$\Ass(R)=\Max(R).$$
    
    We prove $(i)\aR{}(ii)$. Call $\mathfrak{m}=\Ann_R(\C(f))\in \Max(R)$. Then $\mathfrak{m}\C(f)=\0$, whence $\C(f)\cont \Ann_R(\mathfrak{m})$.

    We prove $(ii)\aR{}(i)$. Note that $\C(f)\mathfrak{m}=\0$, whence $\mathfrak{m}\cont \Ann_R(\C(f))$. Since $f\neq 0$, we have that $\C(f)\neq 0$ and thus that $\Ann_R(\C(f))$ is a proper ideal. By maximality of $\mathfrak{m}$, we conclude that $\Ann_R(\C(f))=\mathfrak{m}\in \Ass(R)$.
\end{proof}

\subsection{Gorenstein factors and modular coefficients}\label{subsec:gorenstein}

Assuming that the ring $R$ is Gorenstein Artinian local, condition $(ii)$ in \corx\ref{corollartinianlocal} simplifies more. This is thanks to the fact that in a Gorenstein Artinian local ring the socle is a 1-dimensional vector space over the residue field; see \cite[Proposition 21.5(c), p.\ 526]{Eisenbud1995}. The family of socle-content polynomials from
Corollary~\ref{corollartinianlocal} can therefore be replaced by a single principal ideal.

\begin{theorem}\label{theorgorensteinartinianlocal}
    Let $R$ be a Gorenstein Artinian local ring, and call $\mathfrak{m}$ its unique maximal ideal. Call then $s$ the generator that exhibits $\Soc(R)$ as a 1-dimensional vector space over the residue field $\kappa = R/\mathfrak{m}$.
    
    Consider $E$ an ideal of $P = R[\{x_\lambda\}]$. The following are equivalent:
    \begin{enumT}
        \item $E$ is an essential ideal of $P$;
        \item $E \cap s P \neq \0$.
    \end{enumT}
\end{theorem}
\begin{proof}
    Of course $(i)\aR{}(ii)$. We prove the converse. By \corx\ref{corollartinianlocal}, it suffices to prove that $E \cap (f) \neq \0$ for all $f\neq 0$ such that $\C(f)\cont \Ann_R(\mathfrak{m})=\Soc(R)$. 
    
    So consider such an $f$. Notice that $\Soc(R)$ is generated by $s$ as an ideal of $R$. This is thanks to the fact that the scalar multiplication of $\Soc(R)$ as a vector space over the residue field $R/\mathfrak{m}$ is induced by the scalar multiplication of $\Soc(R)$ as an $R$-module, using that $\mathfrak{m}\Soc(R)=\0$. Since $\C(f)\cont \Soc(R)=(s)$, all coefficients of $f$ are multiples of $s$. Thus, $f\in sP$; say that $f=s\cdot h$ with $h\in P$. Then, by assumption, we can take a non-zero element $e\in E\cap sP$. This means that there exists $g\in P$ such that $e=s\cdot g$, and $s\cdot g\neq 0$. 
    
    We claim that $0\neq eh\in E\cap (f)$. Of course $eh\in E$ since $E$ is an ideal. Moreover, 
    $$eh=sgh=gf\in (f).$$
    It remains to prove that $eh\neq 0$.

    For this, we look at the classes of $g$ and $h$ in the quotient of $P$ modulo $\mathfrak{m}P$. We have that
    $$P/\mathfrak{m}P\iso (R/\mathfrak{m})[\{x_\lambda\}]=\kappa[\{x_\lambda\}],$$
    whence $P/\mathfrak{m}P$ is an integral domain (since $\kappa$ is a field) and the projection $P\to P/\mathfrak{m}P$ coincides with taking the classes in $R/\mathfrak{m}$ of all coefficients of the starting polynomial. We observe that $\overline{g}\neq 0$ in $P/\mathfrak{m}P$. Indeed, using that $\Ann_R(\mathfrak{m})=\Soc(R)=(s)$ and thus $s\mathfrak{m}=\0$, the fact that $s\cdot g\neq 0$ guarantees that there is at least one coefficient of $g$ that is not in $\mathfrak{m}$. Similarly, the fact that $s\cdot h=f\neq 0$ guarantees that $\overline{h}\neq 0$ in $P/\mathfrak{m}P$. Since $P/\mathfrak{m}P$ is an integral domain, we then obtain that $\overline{g}\overline{h}=\overline{gh}\neq 0$ in $P/\mathfrak{m}P$. But this means that $gh$ has at least one coefficient that is not in $\mathfrak{m}$. Finally, notice that $\mathfrak{m}=\Ann_R(s)$, since surely $\mathfrak{m}\cont \Ann_R(s)$, $\Ann_R(s)$ is a proper ideal and $\mathfrak{m}$ is maximal. So $gh$ has at least one coefficient that is not in $\Ann_R(s)$, and then $eh=sgh\neq 0$.
\end{proof}

We would now like to extend the result of \thex\ref{theorgorensteinartinianlocal} to the case of $R$ Gorenstein Artinian but not necessarily local. We will use the well-known theorem of decomposition of Artinian rings as finite products of Artinian local rings (see \cite[Theorem 8.7, p.\ 90]{AtiyahMacdonald1969}).

To pass from local to general Artinian Gorenstein coefficients, we record the
behaviour of essentiality under finite products.

\begin{lemma}\label{lemfiniteproductdecomp}
    Let $R$ be a commutative ring that decomposes as a finite product of rings
    $$R\iso R_1\times \cdots \times R_n.$$
    Then $P = R[\{x_\lambda\}]$ decomposes as
    $$P=R_1[\{x_\lambda\}]\times \cdots \times R_n[\{x_\lambda\}].$$
    
    Consider $E$ an ideal of $P = R[\{x_\lambda\}]$. And consider the canonical orthogonal idempotents $e_i=(0,\ldots, 1,\ldots 0)$ with $1$ at position $i$, seen in $R$. Note that these form a complete set of orthogonal idempotents $\{e_1,\ldots e_n\}$ of $R$. Then the following are equivalent:
    \begin{enumT}
        \item $E$ is an essential ideal of $P$;
        \item the projection $E_i:= e_i E$ of $E$ to the $i$-th factor is an essential ideal of $R_i[\{x_\lambda\}]$ for every $i\in \{1,\ldots n\}$.
    \end{enumT}
\end{lemma}
\begin{proof}
    The product decomposition of $P$ is given by \cite[Chapter IV, §1, no.\ 1, p.\ IV.1]{Bou90}. Notice that $P_i:=R_i[\{x_\lambda\}]\iso e_i P$.

    We prove $(i)\aR{}(ii)$. Let $I_i$ be a non-zero ideal of $P_i$. We can view $I_i$ as an ideal of $P$, and by assumption $E\cap I_i\neq \0$. Take a non-zero element $x\in E\cap I_i$. Since $x\in I_i$, $x$ must be a multiple of $e_i$. But then we obtain that $x\in e_i E\cap I_i$. We conclude that $E_i$ is an essential ideal of $P_i$.

    We prove $(ii)\aR{}(i)$. Let $I$ be a non-zero ideal of $P$. We need to prove that $E\cap (I)\neq 0$. Take a non-zero element $x\in I$ and decompose it as $x=e_1 x+\ldots +e_n x$. Since $x\neq 0$, there exists $k\in \{1,\ldots,n\}$ such that $e_k x\neq 0$. Now, consider $I_k:=e_k I\cont I$. We have that $I_k$ is a non-zero ideal of $P_k$. Since $E_k$ is an essential ideal of $P_k$ by assumption, we conclude that $E_k\cap I_k\neq \0$. Whence
    $$E\cap I\contain E_k\cap I_k\neq 0.$$
    Therefore $E$ is an essential ideal of $P$.
\end{proof}

\begin{remark}\label{remgenofsocle}
    As we saw in the proof of \thex\ref{theorgorensteinartinianlocal}, the socle of a Gorenstein Artinian local ring is generated, as an ideal, by a single element.
    
    When we consider $R$ a Gorenstein Artinian ring that is not necessarily local, this result is still true. Indeed, it is well-known that every Gorenstein Artinian ring $R$ decomposes as a finite product
    $$R\iso R_1\times \cdots \times R_n$$
    of Gorenstein Artinian local rings (see \cite[Theorem 8.7, p.\ 90]{AtiyahMacdonald1969} and \cite[Definition 3.1.18]{BrunsHerzog1998}). Call $\{e_1,\ldots, e_n\}$ the complete set of orthogonal idempotents of $R$ built as in \lemx\ref{lemfiniteproductdecomp}. Observe that
    $$\Soc(R)\iso \Soc(R_1)\times \cdots \times \Soc(R_n)$$
    by \cite[§6, Exercise 12(6), p.\ 242]{Lam99}. For every $i\in \{1,\ldots, n\}$, there exists an element $s_i\in R_i$ that generates $\Soc(R_i)$, as an ideal of $R_i$. Then, we can construct a global element $s=e_1 s_1+\ldots +e_n s_n\in R$. And $s$ generates $\Soc(R)$ in $R$. Furthermore, $s$ is such that $e_i s=s_i$ generates $\Soc(R_i)$.
\end{remark}

\begin{corollary}\label{corgorensteinartinian}
    Let $R$ be a Gorenstein Artinian ring, and consider a decomposition of $R$ as a finite product
    $$R\iso R_1\times \cdots \times R_n$$
    of Gorenstein Artinian local rings (see \cite[Theorem 8.7, p.\ 90]{AtiyahMacdonald1969} and \cite[Definition 3.1.18]{BrunsHerzog1998}). Call $\{e_1,\ldots, e_n\}$ the complete set of orthogonal idempotents of $R$ built as in \lemx\ref{lemfiniteproductdecomp}. Then, call $s$ the generator of $\Soc(R)$ constructed in \remx\ref{remgenofsocle}.
    
    Consider $E$ an ideal of $P = R[\{x_\lambda\}]$. The following are equivalent:
    \begin{enumT}
        \item $E$ is an essential ideal of $P$;
        \item $E \cap e_i s P \neq \0$ for every $i \in \{1, \ldots, n\}$.
    \end{enumT}
\end{corollary}
\begin{proof}
    Of course $(i)\aR{}(ii)$. We prove the converse. By \lemx\ref{lemfiniteproductdecomp}, it suffices to check that $e_i E$ is an essential ideal of $P_i:=R_i[\{x_\lambda\}]$ for every $i\in \{1,\ldots, n\}$. But each ring $R_i$ is Gorenstein Artinian local, and its socle $\Soc(R_i)$ is generated by $e_i s$, thanks to \remx\ref{remgenofsocle}.

    So, applying \thex\ref{theorgorensteinartinianlocal}, it suffices to check that $e_i E\cap e_i s P_i\neq \0$. But $e_i s P_i=e_i s P$, and
    $$E\cap e_i s P = e_i E\cap e_i s P.$$
    Whence we conclude.
\end{proof}

The result we obtained above for $R$ Gorenstein Artinian can be applied to the case of $R=\Z_n=\Z/n\Z$. Indeed, it is well-known that $\Z_n$ is Gorenstein Artinian (see \cite[Corollary 21.19]{Eisenbud1995}). The resulting criterion depends only on
the distinct prime divisors of $n$.

\begin{corollary}\label{cormodularnumbers}
    Consider $E$ an ideal of $P = \Z_n[\{x_\lambda\}]$ (with $n>1$). The following are equivalent:
    \begin{enumT}
        \item $E$ is an essential ideal of $\Z_n[\{x_\lambda\}]$;
        \item $E \cap \left(\frac{n}{p_i}\right)P \neq \0$ for every prime $p_i$ that divides $n$.
    \end{enumT}
\end{corollary}
\begin{proof}
    Consider the prime factorization of $n$. Say that
    $$n = p_1^{k_1} p_2^{k_2} \cdots p_m^{k_m}.$$
    By the Chinese Remainder theorem, $\Z_n$ decomposes as a finite product of Gorenstein Artinian local rings as follows:
    $$\Z_n \iso \Z_{p_1^{k_1}} \times \dots \times \Z_{p_m^{k_m}}$$
    The fact that $\Z_{p_i^{k_i}}$ is local is well-known, see for example \cite[Corollary 21.19, p.\ 537]{Eisenbud1995}.
    
    Then, we have that $\Soc(R)$ is generated by
    $$s = \frac{n}{p_1 p_2 \cdots p_m}$$
    This result can be found in \cite[Section 9.1, p.\ 218]{Kas82}.
    
    The idempotents $e_i$ described in \thex\ref{corgorensteinartinian} are such that
    $$\begin{cases}
        e_i = 1 \;\; (\text{mod } p_i^{k_i})\\ 
        e_i = 0 \;\; (\text{mod } p_j^{k_j})\; \text{ for every } j\neq i
    \end{cases}$$
    These can be calculated by the Chinese Remainder theorem. A common technique employs the B\'ezout identity as follows. We know that there exist integers $u_i$ and $v_i$ such that
    $$u_i \frac{n}{p_i^{k_i}}+v_i p_i^{k_i}=1.$$
    Then
    $$e_i=u_i\frac{n}{p_i^{k_i}} \;\; (\text{mod } n).$$
    But, actually, we do not need to fully calculate these idempotents $e_i$.

    We claim that the ideal generated by $e_i s$ in $\Z_n$ coincides with the ideal generated by $\frac{n}{p_i}$. Indeed, we see that
    $$e_i s = u_i \frac{n}{p_i^{k_i}}p_i^{k_i-1}\Prod{j\neq i} p_j^{k_j-1}$$
    is a multiple of $\frac{n}{p_i^{k_i}}p_i^{k_i-1}=\frac{n}{p_i}$. Moreover, starting from $u_i \frac{n}{p_i^{k_i}}+v_i p_i^{k_i}=1$, we obtain that
    $$u_i \frac{n}{p_i^{k_i}}\frac{n}{p_i}+v_i p_i^{k_i}\frac{n}{p_i}=\frac{n}{p_i}$$
    Modulo $n$, this simplifies to
    $$\frac{n}{p_i}=e_i \frac{n}{p_i}=e_i s \Prod{j\neq i} p_j.$$
    Therefore, $(e_i s)=\left(\frac{n}{p_i}\right)$. 

    We conclude by applying \corx\ref{corgorensteinartinian}.
\end{proof}

\section{Completions of excellent rings}\label{sec:excellent}

Finally, we investigate the case in which the ring $R$ is an ideal-adic completion of an excellent ring $A$. The condition that $A$ is excellent guarantees that the completion is well-behaved, and allows us to just check minimal primes of $\widehat{A}$ over the extensions of associated primes of $A$.

\begin{theorem}\label{theorexcellentcompletions}
    Let $R=\widehat{A}^I$ be the completion of an excellent ring $A$ with respect to an ideal $I$ of $A$. Consider $E$ an ideal of $P = \widehat{A}^I[\{x_\lambda\}]$. The following are equivalent:
    \begin{enumT}
        \item $E$ is an essential ideal of $P$;
        \item $E \neq \0$ and $E \cap (f) \neq \0$ for all $f\neq 0$ such that $\Ann_{\widehat{A}^I}(\C(f)) = \mathfrak{P}$ with $\mathfrak{P}$ a minimal prime ideal in $\widehat{A}^I$ over the extension $\mathfrak{p}\widehat{A}^I$ of some $\mathfrak{p}\in \Ass(A)$ such that $\mathfrak{p}+I\neq A$.
    \end{enumT}
\end{theorem}
\begin{proof}
    Since $A$ is Noetherian, it is well-known that the completion $\widehat{A}^I$ of $A$ is Noetherian as well (see \cite[Chapter 10, Theorem 10.26, p.\ 113]{AtiyahMacdonald1969}). Then, by \thex\ref{theornoetherian}, $E$ is an essential ideal of $P$ if and only if $E\neq \0$ and $E\cap (g) \neq \0$ for all $g\neq 0$ such that $\Ann_{\widehat{A}^I}(\C(g)) \in \Ass(\widehat{A}^I)$.

    We claim that $\Ass(\widehat{A}^I)$ precisely coincides with the set of all minimal prime ideals in $\widehat{A}^I$ over the extension $\mathfrak{p}\widehat{A}^I$ of some $\mathfrak{p}\in \Ass(A)$ such that $\mathfrak{p}+I\neq A$. That is,
    $$\Ass(\widehat{A}^I) = \bigcup_{\substack{\mathfrak{p} \in \Ass(A) \\ \mathfrak{p} + I \neq A}} \Min(\widehat{A}^I / \mathfrak{p}\widehat{A}^I)$$
    where we slightly abuse the notation to identify the prime ideals of $\widehat{A}^I / \mathfrak{p}\widehat{A}^I$ with the prime ideals of $\widehat{A}^I$ that contain $\mathfrak{p}\widehat{A}^I$. After proving this, it will be immediate to conclude.
    
    To prove the claim, we use the fact that the $I$-adic completion map $\phi: A \to \widehat{A}^I$ is a regular ring homomorphism, thanks to the fact that $A$ is excellent. References for this are \cite[§33.I, Theorem 79]{Matsumura1986} and \cite[5–231, Scholie (7.8.3)(v), p.\ 215]{Gro65}. Since $\phi$ is in particular flat, by \cite[§23, Theorem 23.2(ii), p.\ 180.]{Matsumura1986} the associated primes of the extension ring $\widehat{A}^I$ are regulated by the associated primes of the base ring $A$ as follows:
    $$\Ass(\widehat{A}^I) = \bigcup_{\mathfrak{p} \in \Ass(A)} \Ass(\widehat{A}^I / \mathfrak{p}\widehat{A}^I).$$
    Moreover, we show that the associated primes $\mathfrak{p} \in \Ass(A)$ that are comaximal with $I$, i.e.\ such that $\mathfrak{p} + I = A$, do not contribute to the associated primes of $\widehat{A}^I$. For this, let $\mathfrak{p}$ be a prime ideal of $A$ such that $\mathfrak{p} + I = A$. We can then write $1=p+x$ with $p\in \mathfrak{p}$ and $x\in I$; whence $p=1-x$. By \cite[Lemma 10.96.6(4), Tag 05GI]{Sta26} (see also \cite[Theorem 8.1(i)]{Matsumura1986}), $x\in I\widehat{A}^I\cont J$ where $J$ is the Jacobson radical of $\widehat{A}^I$. As a consequence, $p=1-x$ must be a unit in $\widehat{A}^I$. Whence $\mathfrak{p}\widehat{A}^I=\widehat{A}^I$. The quotient $\widehat{A}^I/\mathfrak{p}\widehat{A}^I$ thus becomes the zero ring, and has no associated primes.

    It remains to prove that for every $\mathfrak{p}\in \Ass(A)$
    $$\Ass(\widehat{A}^I / \mathfrak{p}\widehat{A}^I)=\Min(\widehat{A}^I / \mathfrak{p}\widehat{A}^I).$$
    So let $\mathfrak{p}\in \Ass(A)$. Since $\phi\:A\to \widehat{A}^I$ is a regular homomorphism between Noetherian rings, and the property of being a regular homomorphism is preserved under base change (see \cite[Lemma 33.4, pp.\ 251–252]{Matsumura1986} and also \cite[Theorem 16.12.1, Tag 07GC, and Lemma 15.42.5, Tag 07EP]{Sta26}), we have that
    $$A/\mathfrak{p}\otimes_A \phi\: A/\mathfrak{p}\too A/\mathfrak{p}\otimes_A \widehat{A}^I\iso \widehat{A}^I / \mathfrak{p}\widehat{A}^I$$
    is a regular homomorphism. For this, notice also that $\widehat{A}^I / \mathfrak{p}\widehat{A}^I$ is Noetherian, since $\widehat{A}^I$ is Noetherian.

    By \cite[Lemma 15.43.1, Tag 07QK]{Sta26}, since $A/\mathfrak{p}$ is reduced and $A/\mathfrak{p}\otimes_A \phi$ is a regular homomorphism, $\widehat{A}^I/ \mathfrak{p}\widehat{A}^I$ must be reduced as well. So $\widehat{A}^I/ \mathfrak{p}\widehat{A}^I$ is a Noetherian reduced ring. As also recalled in Subsection~\ref{subsec:s1}, \mbox{Noetherian} reduced rings satisfy Serre's condition (S1). Whence the associated primes of $\widehat{A}^I/ \mathfrak{p}\widehat{A}^I$ coincide with the minimal primes.    
\end{proof}

\end{document}